\documentclass[12pt]{article}
\usepackage{amsmath, amsthm, amscd, amsfonts, amssymb, graphicx, color}
\usepackage{epstopdf}
\usepackage{subfigure}
\usepackage{setspace}
\usepackage[left=1in, right=1in, bottom=1.4in]{geometry}
\usepackage{bm}
\usepackage{cite}
\newtheorem{theorem}{Theorem}[section]

\newtheorem{lemma}[theorem]{Lemma}

\newtheorem{proposition}[theorem]{Proposition}
\newtheorem{remark}{Remark}

\title{ \bf Some coefficient estimates on  complex-valued kernel $\alpha-$harmonic mappings$^{\ast}$ }
\date{}
 \begin{document}

\maketitle
\renewcommand{\thefootnote}{\fnsymbol{footnote}}
\noindent{\footnotesize\rm{ \begin{center}\bf Bo-Yong Long \end{center}\
 \begin{center}  School of Mathematical Sciences, Anhui University, Hefei  230601, China
\end{center}}
\footnotetext{\hspace*{-5mm}
\begin{tabular}{@{}r@{}p{16.0cm}@{}}
&$^{*}$ Supported by   NSFC (No.12271001) and  Natural Science Foundation of Anhui Province
(2308085MA03),  China.
\\
&E-mail:  boyonglong@163.com;

\end{tabular}}

%\def\abstractname{}
%\begin{abstract}
\begin{quote}
{\small \noindent {\bf Abstract}:\   We call  a kind of mappings induced by a kind of weighted Laplace operator   as complex-valued kernel $\alpha-$harmonic mappings. In this article, for this class of mappings, the Heinz type lemma is established, and the best Heinz type inequality is obtained. Next, the extremal function of Schwartz's Lemma is discussed. Finally, the coefficients are estimated for the subclass of complex-valued kernel $\alpha-$harmonic mappings whose coefficients are real numbers.
}\\
  {{\small \noindent{\bf Keywords}:  Weighted harmonic mapping; Heinz's type lemma; Schwarz lemma; typically real; coefficient estimates   }
 }
\end{quote}

\section{Introductions and main results}
\setcounter{equation}{0}
\hspace{2mm}

\hspace{2mm}

Let $\mathbb{C}$ be the complex plane, $\mathbb{D}$ the unit disk  and $\mathbb{T}=\partial\mathbb{D}$ the boundary of unit disk.

A kind of weighted Laplace operator be given by
$$\Delta_{\alpha, z}=\frac{\partial}{\partial z}(\omega_{\alpha})^{-1}\frac{\partial}{\partial \bar{z}}=\frac{\partial}{\partial z}(1-|z|^{2})^{-\alpha}\frac{\partial}{\partial \bar{z}},$$
where $\alpha>-1$, $\omega_{\alpha}=(1-|z|^{2})^{\alpha}$ is the so-called standard weight.
The corresponding homogeneous partial differential equation is
\begin{align}\label{1.1}
\Delta_{\alpha,z}(u)=0,  \quad \mbox {in} \quad \mathbb{D}.
\end{align}
Denote the associated  Dirichlet boundary value problem as follows
\begin{equation} \;  \left\{
\begin{array}{rr}
\Delta_{\alpha, z}(u)=0 &\quad \mbox {in}  \,\mathbb{D}, \\
u=u^{*} &\quad \mbox {on}  \,\mathbb{T}.
\end{array}\right.
\end{equation}
Here, the boundary data $u^{*}\in\mathfrak{D}'(\mathbb{T})$
is a distribution on $\mathbb{T}$, and the boundary condition in (1.2) is interpreted in the distributional sense that $u_{r}\rightarrow u^{*}$ in $\mathfrak{D}'(\mathbb{T})$
as $r\rightarrow1^{-}$, where
\begin{equation*}
u_{r}(e^{i\theta})=u(re^{i\theta}), \quad e^{i\theta}\in\mathbb{T},
\end{equation*}
for $r\in[0,1)$.

In \cite{Olofsson2013}, the authors proved that, for parameter $\alpha>-1$, if a function $u\in\mathcal{C}^{2}(\mathbb{D})$ satisfies (1.1) with $\lim_{r\rightarrow 1^{-}}u_{r}=u^{*}\in\mathfrak{D}'(\mathbb{T})$
, then it has the form of Poisson type integral
\begin{equation}\label{1.3}
u(z)=\frac{1}{2\pi}\int_{0}^{2\pi}\mathcal{P}_{\alpha}(ze^{-i\tau})u^{*}(e^{i\tau})d\tau, \quad \, z\in\mathbb{D},
\end{equation}
where
 \begin{equation}\label{1.4}
 \mathcal{P}_{\alpha}(z)=\frac{(1-|z|^{2})^{\alpha+1}}{(1-z)(1-\bar{z})^{\alpha+1}}.
 \end{equation}
Observe that the kernel $\mathcal{P}_{\alpha}$ in (\ref{1.4}) is complex-valued.
 We  call the solution $u$  of  the equation (\ref{1.1}) with $\alpha>-1$ {\bf complex-valued kernel $\alpha-$harmonic mappings (or functions)}.

 \begin{theorem} \cite{Olofsson2013} A function $u$ in $\mathbb{D}$ is a complex-valued kernel $\alpha$-harmonic mapping if and only if it is given
by a convergent power series expansion of the form

 \begin{equation}\label{1.5}
 u(z)=\sum_{k=0}^{\infty}c_{k}z^{k}+\sum_{k=1}^{\infty}c_{-k}P_{\alpha, k}(|z|^{2})\bar{z}^{k}, \quad z\in \mathbb{D},
 \end{equation}
 for some sequence $\{c_{k}\}_{-\infty}^{\infty}$ of complex number satisfying
 \begin{equation*}
 \lim_{|k|\rightarrow\infty}\sup|c_{k}|^{\frac{1}{|k|}}\leq 1,
 \end{equation*}
 where  \begin{equation}\label{1.6}
 P_{\alpha, k}(x)=\int^{1}_{0}t^{k-1}(1-tx)^{\alpha}dt, \quad -1<x<1. \end{equation}

 \end{theorem}

In the Remark 1.3 of \cite{Olofsson2013}, the authors pointed out that the complex-valued kernel $\alpha-$harmonic mapping $u$ with power series expansion  (1.5) actually converges in $\mathcal{C}^{\infty}(\mathbb{D})$.

If we take $\alpha=p-1$ is a nonnegative integer,  then from Theorem 1.1 of \cite{Chenxing2016},  we find that the complex-valued kernel $\alpha-$harmonic mapping is polyharmonic (exactly, $p-$harmonic). As to the general polyharmonic mappings, we refer to \cite{Abdulhadi2008,  Chenshao2011,  Pavlovc1997, Chenjiao2015, Qiao2016, Amozova2017, Lipei2013}.
In particular, if $\alpha=0$, then the complex-valued kernel $\alpha-$harmonic mapping is harmonic. As to the general harmonic mappings, we refer to \cite{Duren2004}.  Thus, the complex-valued kernel $\alpha-$harmonic mappings is a kind of generalization of classical harmonic mappings.    Furthermore, by \cite{Olofsson2013, Olofsson2014}, we know that complex-valued kernel $\alpha-$harmonic mappings are related to the real kernel $\alpha-$harmonic mappings. For the related study on real kernel $\alpha-$harmonic mappings, see \cite{ Chenshao2015, Lipei2018, MR4365391}.

 For the complex-valued kernel $\alpha-$harmonic mappings, the representation theorem  and Lipschitz continuity  are studied in \cite{Chenxing2016} when $\alpha$ is  a nonnegative integer;   the Lipschitz continuity  for general $\alpha>-1$  are considered in \cite{Lipei2019} and \cite{Olofsson2020}; the Schwarz-Pick type inequality and coefficient estimates are obtained in \cite{Lipei2017}; the starlikeness,  convexity and Rad\'{o}-Kneser-Choquet type theorem are explored in \cite{Longbo2021}.  The corresponding case of  complex-valued kernel $\alpha-$harmonic mappings defined  on the upper half plane are researched in
 \cite{ Carlsson2016 }.

 Similar to the definition of sense-preserving harmonic mappings \cite{Duren1996}, we define the sense-preserving complex-valued  kernel $\alpha-$harmonic mappings. A complex-valued kernel $\alpha-$harmonic mapping $u$, not identically constant, is called  sense-preserving at a point $z_{0}$ if $u_{z}\not\equiv  0$, $\omega:=\overline{u}_{z}/u_{z}$ is well defined at $z_{0}$ (possibly with a removable singularity), and $|\omega(z_{0})|<1$. If $u$ is sense-preserving at $z_{0}$,  then either the Jacobian $J_{u}(z_{0})=|u_{z}(z_{0})|^{2}-|u_{\bar{z}}(z_{0})|^{2}=|u_{z}(z_{0})|^{2}-|\overline{u}_{z}(z_{0})|^{2}>0$ or $u_{z}(z_{0})=0$ for an isolated point $z_{0}$. We say $u$ is sense-preserving in $\mathbb{D}$ if $u$ is sense-preserving at all $z\in\mathbb{D}$.

In the harmonic mapping theory, Heinz's lemma is an important and fundamental result. It has applications in many aspects including minimal surface theory.
For complex-valued  kernel $\alpha-$harmonic self-mappings of the unit disk, we establish a  Heinz's type lemma which is as follows:

 \begin{theorem}\label{theorem1.2}(Heinz's type lemma)\quad
Let $u(z)$ be  a  complex-valued kernel $\alpha-$harmonic mapping of the form (\ref{1.5}). Suppose $u(z)$ maps the unit disk $\mathbb{D}$ onto itself. Then
\begin{align}\label{1.7}
|c_{1}|^{2}
+\frac{3\sqrt{3}}{\pi}|c_{0}|^{2}
+\frac{1}{(\alpha+1)^{2}}|c_{-1}|^{2}
\geq \frac{27}{4\pi^{2}}.
\end{align}The lower bound $\frac{27}{4\pi^{2}}$  is sharp.
\end{theorem}

It is worth pointing out that if $\alpha\in(-1,0)$, the image domain of the extremal function  in the Theorem \ref{theorem1.2} is a relatively compact subset of the unit disc $\mathbb{D}$. See Remark \ref{remark1} in Section 2 for the detail.
If $\alpha=0$, then  Theorem \ref{theorem1.2} reduces to the case of harmonic mappings, see \cite{Duren2004}.

Define $\Lambda_{u}:=|u_{z}|+|u_{\bar{z}}|$ and $\lambda_{u}:=||u_{z}|-|u_{\bar{z}}||$. We get the following conclusion that involves the extremal function of Schwarz type lemma of the class of  complex-valued  kernel $\alpha-$harmonic mappings.
 \begin{theorem}\label{theorem1.3}
Suppose $u(z)$ is  a sense-preserving  complex-valued kernel $\alpha-$harmonic mapping of the unit disk $\mathbb{D}$,  $u(0)=0$ and $|u(z)|<1$.
If $J_{u}(0)=1$ or $\lambda_{u}(0)=1$,
then $u(z)=\beta z$  with $|\beta|=1$.

\end{theorem}

If $\alpha=0$, then  Theorem \ref{theorem1.3} reduces to the case of harmonic functions, see \cite{LiuM2009}.

According to Theorem \ref{theorem1.3} and its proof, we obtain the following proposition.

\begin{proposition}\label{proposition1.4}
There does not exist a  sense-preserving  complex-valued  kernel $\alpha-$harmonic mapping $u(z)$ on $\mathbb{D}$ with $u(0)=0$ and $|u(z)|<1$, such that $J_{u}(0)>1$.
\end{proposition}

 Let $H_{\alpha}$ be the family of all  complex-valued kernel $\alpha-$harmonic mappings  with the form (\ref{1.5}),  $S_{H_{\alpha}}$ be the subclass of   $H_{\alpha}$  consisting of $u$  that $u$ are univalent and sense-preserving in $\mathbb{D}$ and satisfy the normalization condition
$u(0)=0$ and $c_{1}=1$. Observe that if $u\in S_{H_{\alpha}}$,
 then $J_{u}(0)>0$ implies that $|c_{-1}|<1$, where
$J_{u}$ is the Jacobian of $u$.   We denote by $S^{o}_{H_{\alpha}}$  the subclass of   $ S_{H_{\alpha}}$  consisting of $u$ such that $c_{-1}=0$.

In \cite{Dorff2012, WangX2003,Bshouty1996}, the typically real harmonic mappings are studied.
Similar to the typically real functions of analytic functions or harmonic functions, we want to study the complex-valued  kernel $\alpha-$harmonic functions with real coefficients.
Let $TS_{H_{\alpha}}$ denote the subclass of $S_{H_{\alpha}}$ consisting of $u$ such that all the coefficients $c_{k}$ and $c_{-k}$ are real numbers,  $TS^{o}_{H_{\alpha}}$ be the subclass of $TS_{H_{\alpha}}$ with $c_{-1}=0$.
 If $u\in TS_{H_{\alpha}}$, then  we call $u$ a
typically real complex-valued  kernel $\alpha-$harmonic function.

We can get the coefficient estimation for $TS_{H_{\alpha}}$ as follows.

\begin{theorem}\label{theorem1.5}
Let $u(z)$ be in $TS_{H_{\alpha}}$ with the form (\ref{1.5}) and  $k=2,3,4,...$.  Then it holds that
\begin{equation}\label{1.8}
\left|\frac{c_{k}-c_{-k}B(k, \alpha+1)}{1-c_{-1}B(1, \alpha+1)}\right|\leq k
\end{equation} if one of the following two conditions is satisfied:

(1) \quad $\alpha\in [0,+\infty)$;

(2)\quad  $\alpha\in(-1,0)$ and $|c_{-1}|<1+\alpha$.

Here $B(\cdot,\cdot)$  is  Beta function.

\end{theorem}

\section{Proofs and Remarks}
\setcounter{equation}{0}
\hspace{2mm}

 Let $k=1,2,...$, $x=|z|^{2}=r^{2}$, $ P_{\alpha,k}=P_{\alpha,k}(x)$ defined as in (\ref{1.6}) and $P_{\alpha,k}(1)=\lim_{x\rightarrow 1^{-}}P_{\alpha,k}(x)$. Then
  \begin{align}\label{2.1}
P_{\alpha,k}(1)=B(k, \alpha+1)=\frac{\Gamma(\alpha+1)\Gamma(k)}{\Gamma(k+\alpha+1)},\end{align} where $B(\cdot,\cdot)$ and $\Gamma(\cdot)$ are   Beta function and  Gamma function, respectively.
In the rest of this paper, we will use  the notations of $P_{\alpha,k}$ and  $P_{\alpha,k}(1)$.

\begin{lemma}\label{lem2.1}
(1)\quad If $\alpha\in[0,+\infty)$, then $1\geq P_{\alpha,1}>\frac{1}{1+\alpha}$;

(2)\quad If $\alpha\in (-1, 0)$, then $1\leq P_{\alpha,1}<\frac{1}{1+\alpha}$;
\end{lemma}
\begin{proof}
By (\ref{1.6}), we have that $P_{\alpha,1}'(x)=-\alpha\int^{1}_{0}t(1-tx)^{\alpha-1}dt$. Thus, $P_{\alpha,1}'(x)\leq 0$ for given $\alpha\in[0,+\infty)$ and $P_{\alpha,1}'(x)>0$ for given $\alpha\in(-1,0)$.
Therefore, Lemma \ref{lem2.1}  follows from the monotonicity of $P_{\alpha,1}(x)$ and the fact that $P_{\alpha,1}(0)=1$ and $P_{\alpha,1}(1)=\frac{1}{1+\alpha}$.
\end{proof}

\begin{lemma}\label{lem2.2}
Let $g(\alpha)=\frac{\Gamma(\alpha+1)}{(\Gamma(\frac{\alpha}{2}+1))^{2}}$. Then $g(\alpha)<1$ for $\alpha\in(-1,0)$.
\end{lemma}
\begin{proof}
Let $\psi(x)=\Gamma'(x)/\Gamma(x)$ be the digamma function. It is well known that $\psi(x)$ is  strictly increasing on $(0, +\infty)$, cf. \cite{Andrews1999}.
Then we have\begin{align*}\frac{d}{d \alpha}(\log g(\alpha))=\psi(\alpha+1)-\psi(\frac{\alpha}{2}+1))< 0
\end{align*}
for $\alpha\in (-1, 0)$. Observe $\log g(\alpha)$ and $g(\alpha)$ have the same monotonicity.  Therefore, Lemma \ref{lem2.2} follows from the monotonicity of $g(\alpha)$ and the fact that $g(0)=1$.

\end{proof}

 \begin{proof}[\textbf{Proof of Theorem \ref{theorem1.2}}]

Without loss generality, we  assume that $u(z)$ is sense-preserving.
Let   $z=re^{i\varphi}$.
By (\ref{1.5}) and (\ref{2.1}), we have
  \begin{align*}
e^{i\theta(\varphi)}&:=\lim_{r\rightarrow 1^{-}}u(z)\\
&=\sum_{k=0}^{\infty}c_{k}e^{ik\varphi}+\sum_{k=1}^{\infty}c_{-k}B(k, \alpha+1)e^{-ik\varphi}.
  \end{align*}Then  $c_{k}$ and
  $c_{-k}B(k, \alpha+1)$ may be regarded as Fourier coefficients of the boundary function  $e^{i\theta(\varphi)}$,  where $\theta(\varphi)$ is a  continuous nondecreasing function with $\theta(\varphi+2\pi)=\theta(\varphi)+2\pi$.
 It follows that
 \begin{align*}
e^{i\theta(\phi+\varphi)}=\sum_{k=0}^{\infty}c_{k}e^{ik(\phi+\varphi)}+\sum_{k=1}^{\infty}c_{-k}B(k, \alpha+1)e^{-ik(\phi+\varphi)},\\
e^{-i\theta(\phi-\varphi)}=\sum_{k=0}^{\infty}\overline{c_{k}}e^{-ik(\phi-\varphi)}+\sum_{k=1}^{\infty}\overline{c_{-k}}B(k, \alpha+1)e^{ik(\phi-\varphi)}.
\end{align*}
 Using Parseval's relation, direction computation leads to

    \begin{align*}
    \frac{1}{2\pi}\int_{0}^{2\pi}e^{i[ \theta(\phi+\varphi)-\theta(\phi-\varphi)]}d\phi=|c_{0}|^{2}+\sum_{k=1}^{\infty}|c_{k}|^{2}e^{2ik\varphi}+\sum_{k=1}^{\infty}|c_{-k}|^{2}(B(k, \alpha+1))^{2}e^{-2ik\varphi}
     \end{align*}
for arbitrary $\varphi\in\mathbb{R}$.  Taking real parts, we have that

\begin{align}\label{2.2}
   1-2J(\varphi)=|c_{0}|^{2}+\sum_{k=1}^{\infty}\left(|c_{k}|^{2}+|c_{-k}|^{2}(B(k, \alpha+1))^{2}\right)\cos(2k\varphi),
     \end{align}where
\begin{align*}
  J(\varphi)= \frac{1}{2\pi}\int_{0}^{2\pi}\sin^{2}\left(\frac{ \theta(\phi+\varphi)-\theta(\phi-\varphi)}{2}\right)d\phi.
     \end{align*}
Let $N(\varphi)$ be the even function  with period $\pi/3$ that has the values  $\cos^{2}(\pi/3+\varphi)$ in the interval $0\leq \varphi\leq\pi/6$.
Furthermore, let
\begin{align*}
 M(\varphi)=\cos^{2}\varphi-N(\varphi), \quad 0\leq\varphi\leq\pi/2.
     \end{align*}
Then
\begin{align*} \; M(\varphi)= \left\{
\begin{array}{rr}   \cos^{2}\varphi-\cos^{2}(\frac{\pi}{3}+\varphi), &\quad 0\leq\varphi\leq\frac{\pi}{6}, \\
 \cos^{2}\varphi-\cos^{2}(\frac{2\pi}{3}-\varphi), &\quad \frac{\pi}{6}\leq\varphi\leq\frac{\pi}{3}, \\
0,\qquad\qquad &\quad \frac{\pi}{3}\leq\varphi\leq\frac{\pi}{2}.
\end{array}\right.
\end{align*}
By the Fourier expansion
\begin{align*}
 N(\varphi)=\frac{1}{2}-\frac{3\sqrt{3}}{4\pi}+\frac{3\sqrt{3}}{2\pi}\sum_{k=1}^{\infty}\frac{1}{9k^{2}-1}\cos6k\varphi
     \end{align*}
and equation (\ref{2.2}), one has
\begin{align}\label{2.3}
    &\frac{8}{\pi}\int_{0}^{\frac{\pi}{2}}M(\varphi)(1-2J(\varphi))d\varphi\nonumber\\
    =&|c_{1}|^{2}+\frac{3\sqrt{3}}{\pi}|c_{0}|^{2}+|c_{-1}|^{2}(B(1,\alpha+1))^{2}
    -\frac{3\sqrt{3}}{\pi}\sum_{k=1}^{\infty}\frac{1}{9k^{2}-1}\left(|c_{3k}|^{2}+|c_{-3k}|^{2}(B(3k,\alpha+1))^{2}\right)\nonumber\\
    \leq&|c_{1}|^{2}+\frac{3\sqrt{3}}{\pi}|c_{0}|^{2}+|c_{-1}|^{2}(B(1,\alpha+1))^{2}.
     \end{align}
Direct computation leads to
 \begin{align}\label{2.4}
    \frac{8}{\pi}\int_{0}^{\frac{\pi}{2}}M(\varphi)d\varphi=\frac{3\sqrt{3}}{\pi}.
 \end{align}
Furthermore, inequality
\begin{align}\label{2.5}
    \frac{16}{\pi}\int_{0}^{\frac{\pi}{2}}M(\varphi)J(\varphi)d\varphi\leq\frac{3\sqrt{3}}{\pi}-\frac{27}{4\pi^{2}}.
 \end{align}
has been proved in $\S$ 4.4 of \cite{Duren2004}. Therefore, inequality (\ref{1.7}) follows from inequalities (\ref{2.3}), (\ref{2.4}), and (\ref{2.5}).

Next, let us to show that  inequality (\ref{1.7}) is sharp.
We rewrite the Poisson type representation (\ref{1.3}) as
 \begin{align}\label{2.6}
 u(z)=\frac{1}{2\pi}\int_{0}^{2\pi}\mathcal{P}_{\alpha}(ze^{-i\varphi})e^{i\theta(\varphi)}d\varphi,
 \end{align}
where $e^{i\theta(\varphi)}$ is the boundary function.
 By the  Definition 2.1 and Theorem 2.5 of \cite{Olofsson2013}, we have
\begin{align*}\mathcal{P}_{\alpha}(ze^{-i\varphi})
 &=\sum_{k=0}^{\infty}e^{-ik\varphi}z^{k}
 +\sum_{k=1}^{\infty}\frac{1}{B(k, \alpha+1)}P_{\alpha,k}e^{ik\varphi}\bar{z}^{k}.
 \end{align*}
 Thus,  the formulas
\begin{align*}\frac{1}{2\pi}\int_{0}^{2\pi}e^{-ik\varphi}z^{k}e^{i\theta(\varphi)}d\varphi &=c_{k}z^{k}\end{align*}and
\begin{align*}
 \frac{1}{2\pi}\int_{0}^{2\pi}\frac{1}{B(k, \alpha+1)}P_{\alpha,k}e^{ik\varphi}\bar{z}^{k}e^{i\theta(\varphi)}d\varphi &=c_{-k}P_{\alpha,k}\bar{z}^{k}
 \end{align*}
 follow from the Poisson type integral representation (\ref{2.6}) and the series expansion (\ref {1.5}).
It follows that
 \begin{align}\label{2.7}
 c_{k}&=\frac{1}{2\pi}\int_{0}^{2\pi}e^{-ik\varphi}e^{i\theta(\varphi)}d\varphi \end{align}and
  \begin{align}\label{2.8} c_{-k}&=\frac{1}{2\pi}\int_{0}^{2\pi}\frac{1}{B(k, \alpha+1)}e^{ik\varphi}e^{i\theta(\varphi)}d\varphi.
 \end{align}

Now taking
\begin{align*}  \theta(\varphi)= \left\{
\begin{array}{rr}   0, &\quad 0\leq\varphi<\frac{\pi}{3},\\
 \frac{2\pi}{3}, &\quad \frac{\pi}{3}<\varphi<\pi, \\
\frac{4\pi}{3}, &\quad \pi<\varphi<\frac{5\pi}{3},\\
0, &\quad\frac{5\pi}{3}<\varphi<2\pi.
\end{array}\right.
\end{align*}
Denote by $U(z)$ the corresponding complex-valued kernel $\alpha$-harmonic mapping produced by (\ref{2.6}).
Then formulas (\ref{2.7}) and (\ref{2.8}) imply that $$c_{1}=\frac{3\sqrt{3}}{2\pi}, \quad c_{0}=c_{-1}=0.$$
Thus the value $27/4\pi^{2}$ is attained in  inequality (\ref{1.7}) by the function $U(z)$.

\end{proof}

\begin{remark}\label{remark1} (1). We don't know how to write down  the explicit expression of the extremal function $U(z)$.

(2). Formula (\ref{2.6}) implies that
 \begin{align*}
 |U(z)|\leq ||e^{i\theta(\varphi)}||_{\infty}\frac{1}{2\pi}\int_{0}^{2\pi}|\mathcal{P}_{\alpha}(ze^{-i\varphi})|d\varphi
 =\frac{1}{2\pi}\int_{0}^{2\pi}|\mathcal{P}_{\alpha}(ze^{-i\varphi})|d\varphi\leq\frac{\Gamma(\alpha+1)}{(\Gamma(\frac{\alpha}{2}+1))^{2}}.
 \end{align*}The last inequality holds   because of formula (0.6) of \cite{Olofsson2013}. If $\alpha\in(-1,0)$, then by Lemma \ref{lem2.2}, we have that $ |U(z)|\leq\frac{\Gamma(\alpha+1)}{(\Gamma(\frac{\alpha}{2}+1))^{2}}<1$ for all $z\in\mathbb{D}$.
That is to say that $U(\mathbb{D})$ is a
relatively compact
subset of $\mathbb{D}$. If $\alpha=0$, in the other words, if it is the classical harmonic case,  then $U(\mathbb{D})$ is an equilateral triangle embedded in the unit circle. See $\S$ 4.2 and $\S$ 4.4 of \cite{Duren2004} for the detail.

\end{remark}

 \begin{proof}[\textbf{Proof of Theorem \ref{theorem1.3}}]
 Let $z=re^{i\theta}$. Then formula (\ref{1.5}) can be rewritten as
\begin{align*}u(re^{i\theta})=\sum_{k=1}^{\infty}c_{k}r^{k}e^{ik\theta}+\sum_{k=1}^{\infty}c_{-k}P_{\alpha,k}r^{k}e^{-ik\theta}
 \end{align*}for $\theta\in[0,2\pi)$. By Parseval's identity and the assumption of $|u(z)|<1$, we obtain

\begin{align*}\sum_{k=1}^{\infty}(|c_{k}|^{2}+|c_{-k}|^{2}(P_{\alpha,k})^{2})r^{2k}=\frac{1}{2\pi}\int^{2\pi}_{0}|u(re^{i\theta})|^{2}d\theta< 1.
 \end{align*}
Letting $r\rightarrow 1^{-}$, we obtain
\begin{align}\label{2.9}
\sum_{k=1}^{\infty}\left(|c_{k}|^{2}+|c_{-k}|^{2}(B(k,\alpha+1))^{2}\right)\leq 1.
 \end{align}Next, we divide it into two cases to discuss.

 (1)\quad If $J_{u}(0)=1$,
then we  have
\begin{align*}|c_{1}|^{2}+|c_{-1}|^{2}(B(1,\alpha+1))^{2}\geq |c_{1}|^{2}-|c_{-1}|^{2}=
J_{u}(0)=1.
\end{align*}Thus inequality (\ref{2.9}) implies that
 \begin{align}\label{2.10}
|c_{1}|=1, \,\, |c_{-1}|=0 \,\,\,\mbox{and} \,\,\,|c_{k}|=|c_{-k}|=0 \,\,\mbox{ for} \,\, k=2,3,....
\end{align}
It follows that
$u(z)=\beta z$, where $|\beta|=1$.

 (2)\quad If $\lambda_{u}(0)=1$, then we have
\begin{align*}|c_{1}|^{2}+|c_{-1}|^{2}(B(1,\alpha+1))^{2}\geq |c_{1}|^{2}-|c_{-1}|^{2}=
J_{u}(0)=\Lambda_{u}(0)\lambda_{u}(0)\geq(\lambda_{u}(0))^{2}=1.
\end{align*}
Thus inequality (\ref{2.9}) still implies that (\ref{2.10}) holds.
 \end{proof}

\begin{proof}[\textbf{Proof of Proposition \ref{proposition1.4}}]
Suppose that there exists a
sense-preserving complex-valued kernel $\alpha-$harmonic mapping  with $u(0)=0$ and $|u(z)|<1$ on $\mathbb{D}$. Then inequality (\ref{2.9}) holds. Specially, it holds that
 \begin{align}\label{2.11}
 |c_{1}|^{2}+|c_{-1}|^{2}(B(1,\alpha+1))^{2}\leq 1.
 \end{align}
If $J_{u}(0)>1$,
then  it holds that\begin{align*}
|c_{1}|^{2}-|c_{-1}|^{2}=
J_{u}(0)>1.
\end{align*} It contradicts the inequality (\ref{2.11}).

\end{proof}

\begin{proof}[\textbf{Proof of Theorem \ref{theorem1.5}}]
Let $z=re^{i\theta}$ and define the difference quotient $\mathcal{D}(z)$ as
\begin{align*} \; \mathcal{D}(re^{i\theta})= \left\{
\begin{array}{rr}   u'(r), &\quad 0<r<1 \,\, \mbox{and}  \,\,\theta=0,\\
\frac{u(z)-u(\overline{z})}{z-\overline{z}}, &\quad 0<r<1 \,\, \mbox{and}  \,\,0<\theta<\pi , \\
u'(-r), &\quad0<r<1 \,\, \mbox{and}  \,\, \theta=\pi.
\end{array}\right.
\end{align*}
Then direct computation leads to
\begin{align}\label{2.12}
\mathcal{D}(re^{i\theta})&=\frac{\sum_{k=1}^{\infty}c_{k}r^{k}e^{ik\theta}
+\sum_{k=1}^{\infty}c_{-k}P_{\alpha,k}r^{k}e^{-ik\theta}-\sum_{k=1}^{\infty}c_{k}r^{k}e^{-ik\theta}
-\sum_{k=1}^{\infty}c_{-k}P_{\alpha,k}r^{k}e^{ik\theta}}{re^{i\theta}-re^{-i\theta}}\nonumber\\
&=\frac{\sum_{k=1}^{\infty}(c_{k}-c_{-k}P_{\alpha,k})r^{k}(e^{ik\theta}-e^{-ik\theta})}{r(e^{i\theta}-e^{-i\theta})}\nonumber\\
&=\sum_{k=1}^{\infty}\left((c_{k}-c_{-k}P_{\alpha,k})r^{k-1}\frac{\sin k\theta}{\sin \theta}\right)
\end{align}
for $0<r<1$ and $0<\theta<\pi$.

If $u(z)$ is univalent, then $\mathcal{D}(re^{i\theta})\neq 0$ for the region $0\leq r< 1$, and $0\leq\theta <2\pi$. Since all the coefficients are real, $\mathcal{D}(z)$ is real.  Furthermore, $\mathcal{D}(0):=\lim_{z\rightarrow 0}\mathcal{D}(z)=1-c_{-1}>0$. Hence, the  continuity of $\mathcal{D}(z)$  and the univalence of $u$ imply that $\mathcal{D}(z)>0$ in $\mathbb{D}\backslash(-1,1)$ and  $\mathcal{D}(z)\geq 0$ in $(-1,1)$.

Let  \begin{align*}I_{k}:=\int^{\pi}_{0}\mathcal{D}(re^{i\theta})\sin\theta \sin k\theta d\theta,\quad k=2,3,4....
\end{align*}On the one hand, by the orthogonality of $\{\sin k\theta\}$ on $[0,\pi]$,

 \begin{align}\label{2.13}
 I_{k}&=\int^{\pi}_{0}\left((1-c_{-1}P_{\alpha,1}+\sum_{n=2}^{\infty}(c_{n}-c_{-n}P_{\alpha,n})r^{n-1}\frac{\sin n\theta}{\sin \theta}\right)\sin\theta \sin k\theta d\theta\nonumber\\
 &=(c_{k}-c_{-k}P_{\alpha,k})r^{k-1}\int^{\pi}_{0}\sin^{2}k\theta d\theta.
\end{align}
On the other hand
\begin{align*}I_{k}=\int^{\pi}_{0}\mathcal{D}(re^{i\theta})\sin\theta \sin k\theta d\theta=\int^{\pi}_{0}\mathcal{D}(re^{i\theta})\sin^{2}\theta \frac{\sin k\theta}{\sin\theta} d\theta.
\end{align*}
If $m_{k}$ and $M_{k}$ denote the minimum and maximum values of $\sin k\theta/\sin\theta$ in $[0,\pi]$, respectively, then
\begin{align}\label{2.14}
m_{k}\int^{\pi}_{0}\mathcal{D}(re^{i\theta})\sin^{2}\theta d\theta\leq I_{k}\leq M_{k}\int^{\pi}_{0}\mathcal{D}(re^{i\theta})\sin^{2}\theta d\theta.
\end{align}
It follows from (\ref{2.13}) and (\ref{2.14}) that
\begin{align*}
m_{k}\int^{\pi}_{0}\mathcal{D}(re^{i\theta})\sin^{2}\theta d\theta\leq (c_{k}-c_{-k}P_{\alpha,k})r^{k-1}\int^{\pi}_{0}\sin^{2}k\theta d\theta\leq M_{k}\int^{\pi}_{0}\mathcal{D}(re^{i\theta})\sin^{2}\theta d\theta.
\end{align*}
Using the series expansion (\ref{2.12}) for $\mathcal{D}(re^{i\theta})$ and the orthogonality of $\{\sin k\theta\}$, we obtain
\begin{align*}m_{k}(1-c_{-1}P_{\alpha, 1})\int^{\pi}_{0}\sin^{2}\theta d\theta\leq (c_{k}-c_{-k}P_{\alpha, k})r^{k-1}\int^{\pi}_{0}\sin^{2}k\theta d\theta\leq M_{k}(1-c_{-1}P_{\alpha,1})\int^{\pi}_{0}\sin^{2}\theta d\theta.
\end{align*}

Now we divide it into two cases to discuss.

{\bf Case (1)} \quad   $\alpha\geq 0$.  Noting the sense-preservity implies that $|c_{-1}|<1$. Then by Lemma \ref{2.1} (1), we have that $1-c_{-1}P_{\alpha, 1}>0$.
Since $\int^{\pi}_{0}\sin^{2}\theta d\theta=\int^{\pi}_{0}\sin^{2}k\theta d\theta=\pi/2$, we  have
\begin{align*}m_{k}\leq \frac{c_{k}-c_{-k}P_{\alpha,k}}{1-c_{-1}P_{\alpha,1}}r^{k-1}\leq M_{k}.
\end{align*} Observe that
\begin{align*}
\lim_{\theta\rightarrow 0^{+}}\frac{\sin k\theta}{\sin \theta}=k, \quad \lim_{\theta\rightarrow \pi^{-}}\frac{\sin k\theta}{\sin \theta}=(-1)^{k+1}k.
\end{align*}We know that $M_{k}=k$, and $m_{k}\geq -k$. But $\min (\sin k\theta/\sin \theta)=-k $ only if $k$ is even. If $k$ is odd, then $m_{k}>-k$. Thus,
\begin{align*}-k\leq \frac{c_{k}-c_{-k}P_{\alpha,k}}{1-c_{-1}P_{\alpha,1}}r^{k-1}\leq k
\end{align*}for all $r\in(0,1)$.
Taking $r=1^{-}$,  we have
\begin{align*}\left|\frac{c_{k}-c_{-k}B(k, \alpha+1)}{1-c_{-1}B(1, \alpha+1)}\right|\leq k.
\end{align*}

{\bf Case (2)} \quad  $\alpha\in (-1, 0)$ and $|c_{-1}|<1+\alpha$. Then Lemma \ref{2.1} (2) still implies that $1-c_{-1}P_{\alpha, 1}>0$. The rest of the proof is same as that of Case (1). We omit it.
\end{proof}

\hspace{2mm}

\hspace{2mm}

{\bf Data availability statement}: Data sharing not applicable to this article as no datasets were generated or analysed during the current study.

\hspace{2mm}

{\bf Competing interests declaration}:
The authors declare that he have no known competing financial interests or personal relationships
that could have appeared to influence the work reported in this paper.

\medskip
\bibliographystyle{plain}
%\bibliographystyle{spphys}
%\begin{thebibliography}{100}
%\bibliographystyle{alpha}
\bibliography{mybib}

\begin{thebibliography}{10}

\bibitem{Abdulhadi2008}
Z.~Abdulhadi and Y.~Abu~Muhanna.
\newblock Landau's theorem for biharmonic mappings.
\newblock {\em J. Math. Anal. Appl.}, 338(1):705--709, 2008.

\bibitem{Amozova2017}
K.~F. Amozova, E.~G. Ganenkova, and S.~Ponnusamy.
\newblock Criteria of univalence and fully {$\alpha$}-accessibility for
  {$p$}-harmonic and {$p$}-analytic functions.
\newblock {\em Complex Var. Elliptic Equ.}, 62(8):1165--1183, 2017.

\bibitem{Andrews1999}
G.~E. Andrews, R.~Askey, and R.~Roy.
\newblock {\em Special functions}, volume~71 of {\em Encyclopedia of
  Mathematics and its Applications}.
\newblock Cambridge University Press, Cambridge, 1999.

\bibitem{Bshouty1996}
D.~Bshouty, W.~Hengartner, and O.~Hossian.
\newblock Harmonic typically real mappings.
\newblock {\em Math. Proc. Cambridge Philos. Soc.}, 119(4):673--680, 1996.

\bibitem{Carlsson2016}
M.~Carlsson and J.~Wittsten.
\newblock The {D}irichlet problem for standard weighted {L}aplacians in the
  upper half plane.
\newblock {\em J. Math. Anal. Appl.}, 436(2):868--889, 2016.

\bibitem{Chenjiao2015}
J.~Chen, A.~Rasila, and X.~Wang.
\newblock Coefficient estimates and radii problems for certain classes of
  polyharmonic mappings.
\newblock {\em Complex Var. Elliptic Equ.}, 60(3):354--371, 2015.

\bibitem{Chenshao2011}
S.~Chen, S.~Ponnusamy, and X.~Wang.
\newblock Bloch constant and {L}andau's theorem for planar {$p$}-harmonic
  mappings.
\newblock {\em J. Math. Anal. Appl.}, 373(1):102--110, 2011.

\bibitem{Chenshao2015}
S.~Chen and M.~Vuorinen.
\newblock Some properties of a class of elliptic partial differential
  operators.
\newblock {\em J. Math. Anal. Appl.}, 431(2):1124--1137, 2015.

\bibitem{Chenxing2016}
X.~Chen and D.~Kalaj.
\newblock A representation theorem for standard weighted harmonic mappings with
  an integer exponent and its applications.
\newblock {\em J. Math. Anal. Appl.}, 444(2):1233--1241, 2016.

\bibitem{Dorff2012}
M.~Dorff, M.~Nowak, and W.~Szapiel.
\newblock Typically real harmonic functions.
\newblock {\em Rocky Mountain J. Math.}, 42(2):567--581, 2012.

\bibitem{Duren2004}
P.~Duren.
\newblock {\em Harmonic mappings in the plane}, volume 156 of {\em Cambridge
  Tracts in Mathematics}.
\newblock Cambridge University Press, Cambridge, 2004.

\bibitem{Duren1996}
P.~Duren, W.~Hengartner, and R.~S. Laugesen.
\newblock The argument principle for harmonic functions.
\newblock {\em Amer. Math. Monthly}, 103(5):411--415, 1996.

\bibitem{Lipei2018}
P.~Li and S.~Ponnusamy.
\newblock Lipschitz continuity of quasiconformal mappings and of the solutions
  to second order elliptic {PDE} with respect to the distance ratio metric.
\newblock {\em Complex Anal. Oper. Theory}, 12(8):1991--2001, 2018.

\bibitem{Lipei2013}
P.~Li, S.~Ponnusamy, and X.~Wang.
\newblock Some properties of planar {$p$}-harmonic and log-{$p$}-harmonic
  mappings.
\newblock {\em Bull. Malays. Math. Sci. Soc. (2)}, 36(3):595--609, 2013.

\bibitem{Lipei2019}
P.~Li and X.~Wang.
\newblock Lipschitz continuity of {$\alpha$}-harmonic functions.
\newblock {\em Hokkaido Math. J.}, 48(1):85--97, 2019.

\bibitem{Lipei2017}
P.~Li, X.~Wang, and Q.~Xiao.
\newblock Several properties of {$\alpha$}-harmonic functions in the unit disk.
\newblock {\em Monatsh. Math.}, 184(4):627--640, 2017.

\bibitem{LiuM2009}
M.~Liu.
\newblock Landau's theorem for planar harmonic mappings.
\newblock {\em Comput. Math. Appl.}, 57(7):1142--1146, 2009.

\bibitem{Longbo2021}
B.~Long and Q.~Wang.
\newblock Some geometric properties of complex-valued kernel
  {$\alpha$}-harmonic mappings.
\newblock {\em Bull. Malays. Math. Sci. Soc.}, 44(4):2381--2399, 2021.

\bibitem{MR4365391}
B.~Long and Q.~Wang.
\newblock Starlikeness, convexity and {L}andau type theorem of the real kernel
  {$\alpha$}-harmonic mappings.
\newblock {\em Filomat}, 35(8):2629--2644, 2021.

\bibitem{Olofsson2014}
A.~Olofsson.
\newblock Differential operators for a scale of {P}oisson type kernels in the
  unit disc.
\newblock {\em J. Anal. Math.}, 123:227--249, 2014.

\bibitem{Olofsson2020}
A.~Olofsson.
\newblock Lipschitz continuity for weighted harmonic functions in the unit
  disc.
\newblock {\em Complex Var. Elliptic Equ.}, 65(10):1630--1660, 2020.

\bibitem{Olofsson2013}
A.~Olofsson and J.~Wittsten.
\newblock Poisson integrals for standard weighted {L}aplacians in the unit
  disc.
\newblock {\em J. Math. Soc. Japan}, 65(2):447--486, 2013.

\bibitem{Pavlovc1997}
M.~Pavlovi\'{c}.
\newblock Decompositions of {$L^p$} and {H}ardy spaces of polyharmonic
  functions.
\newblock {\em J. Math. Anal. Appl.}, 216(2):499--509, 1997.

\bibitem{Qiao2016}
J.~Qiao.
\newblock Univalent harmonic and biharmonic mappings with integer coefficients
  in complex quadratic fields.
\newblock {\em Bull. Malays. Math. Sci. Soc.}, 39(4):1637--1646, 2016.

\bibitem{WangX2003}
X.~Wang, X.~Liang, and Y.~Zhang.
\newblock On harmonic typically real mappings.
\newblock {\em J. Math. Anal. Appl.}, 277(2):533--554, 2003.

\end{thebibliography}

\end{document}